\documentclass[11pt, letterpaper]{amsart}
\usepackage[left=1in,right=1in,bottom=1in,top=1in]{geometry}
\usepackage{graphicx}
\usepackage{amsmath,amsthm,amssymb}
\usepackage{mathrsfs, mathtools}
\usepackage[all,cmtip]{xy}
\usepackage{cite}
\usepackage{setspace}
\usepackage{youngtab}

\newtheorem{thm}{Theorem}[section]
\newtheorem{lemma}[thm]{Lemma}

\newtheorem{prop}[thm]{Proposition}

\DeclareMathOperator{\Var}{Var}
\DeclareMathOperator{\id}{id}
\DeclareMathOperator{\TV}{TV}
\onehalfspace

\def\.{\hskip.06cm}
\def\ts{\hskip.03cm}

\begin{document}

\title[Cutoff for biased transpositions]{Cutoff for biased transpositions}
\author{Megan Bernstein$^\star$ \and Nayantara Bhatnagar$^\dagger$ \and Igor Pak$^\ddagger$}

\thanks{\today}
\thanks{\thinspace ${\hspace{-.45ex}}^\ast$School of Mathematics, Georgia Institute of Technology,
Atlanta, GA~30332.
\hskip.06cm
Email:
\hskip.06cm
\texttt{bernstein@math.gatech.edu}}

\thanks{\thinspace ${\hspace{-.45ex}}^\dagger$Mathematical Sciences,
University of Delaware, Newark, DE, 19716.
\hskip.06cm
Email:
\hskip.06cm
\texttt{naya@math.udel.edu}}

\thanks{\thinspace ${\hspace{-.45ex}}^\ddagger$Department of Mathematics,
UCLA, Los Angeles, CA~90095.
\hskip.06cm
Email:
\hskip.06cm
\texttt{({pak}@)math.ucla.edu}}

\maketitle

\begin{abstract}
In this paper we study the mixing time of a biased transpositions shuffle on a set of $N$ cards with $N/2$ cards of two types. For a parameter $0<a \le 1$, one type of card is chosen to transpose with a bias of $\frac{a}{N}$ and the other type is chosen with probability $\frac{2-a}{N}$. We show that there is cutoff for the mixing time of the chain at time $\frac{1}{2a} N \log N$. Our proof uses a modified marking scheme motivated by Matthews' proof of a strong uniform time for the unbiased shuffle~\cite{Mat}.
\end{abstract}

\bigskip

\section{Introduction}

\noindent
The \emph{cutoff phenomenon} is a remarkable property of Markov chains, indicating rapid transition of the chain from an unmixed to a mixed state.  It originates in the study of \emph{phase transitions} in physics, but remains challenging to establish in many natural situations.  We refer to~\cite{Dia,LPW} for the  introduction to cutoff including many examples and review of the literature, and to~\cite{BHP,CS} for some general results and further examples.

In this paper we consider the case of a random walk on the symmetric group $S_N$ with random transpositions. This is one of the most classical examples motivated by card shuffling.  The permutations are represented as the orderings of $N$ cards on a table. Our generators, the transpositions, as a selection of a card by the right hand and left hand, whose locations will be exchanged.  Diaconis and Shahshahani famously showed in~\cite{DS} that if the right hand and left hands are placed independently uniformly at random, then this walk mixes with cutoff at $\frac{1}{2}N\log N \pm cN$ steps.  The proof is based on character estimates, so the symmetry played a critical role.

In this paper we break the symmetry and analyze the following \emph{biased random walk with two types}.  Consider the random walk on the symmetric group $S_N$ with transpositions. Instead choosing transpositions uniformly at random, the transposition $(i j)$ is chosen with probability $p_{i,j} = p_i p_j$ with the identity chosen with probability $\sum_{i} p_i^2$. Let $n = N/2$, where $N$ is even. We consider the case where the $p_i$ are evenly split between taking values $\frac{a}{N}$ or $\frac{b}{N}$ with $a +b =2$, and $0<a\leq b$.
In the biased scheme, this amounts to a hand landing on a specific $a$ card with probability $\frac{a}{N}$ and a $b$ card with probability $\frac{b}{N}$, again with independence between the hands. We will show here that when the selection of the cards is biased towards half the cards, cutoff will still occur at the delayed time of $\frac{1}{2a}N\log N \pm o(N\log N)$.

This walk, despite the bias, still has uniform stationary distribution. This is because the walk is reversible with respect to the stationary distribution since for each $i,j$, $p_{i,j} = p_{j,i}>0$.  Biasing schemes in which this does not hold, for example where $p_{i,j}$ depends on whether $i$ and $j$ are currently in order, have been studied by~\cite{BMRS} (see also~\cite{Jon}).  Another scheme for random transpositions was studied in~\cite{MPS,Pak}, where one hand has a deterministic behavior.  In none of these cases the cutoff has been established.

\begin{thm}
For all $0< a <2$ and $\varepsilon \in (0,1)$, the mixing time $T(\varepsilon)$ of the biased random walk defined as above, satisfies
$$\lim_{N \to \infty} \, \frac{ 2\ts a \. T(\varepsilon)}{ N \log N} \, = \. 1\ts.
$$
\end{thm}

\noindent
Here the \emph{mixing time} $T(\varepsilon)$ is defined in terms of the \emph{separation distance} (see e.g.~\cite{LPW}): 
$$
T(\varepsilon) \, := \, \min\left\{\. t \, : \, P^t(\sigma) \ge \frac{1-\varepsilon}{N!} \ \, \text{for all} \ \sigma\in S_N\right\}\ts.
$$
The lower bound in the theorem follows from a coupon collector argument. A matching upper bound is achieved by modifying a strong uniform time of Matthews~\cite{Mat}. The \emph{total variation distance} mixing time
$$T_{\TV}(\varepsilon) \, := \, \min\left\{\. t\,: \frac{1}{2}\sum_{\sigma\in S_N} \left|P^t(\sigma)  - \frac{1}{N!}\right| \leq \epsilon \right\}\ts$$ 
satisfies $T_{TV}(\varepsilon) \leq T(\varepsilon)$ (see e.g. ~\cite{LPW}).
Since the lower bound is established in terms of total variation distance and the upper bound in terms of separation distance, the result gives cutoff in both separation distance and total variation distance.

Note that a general strong uniform time argument by the third author~\cite{Pak} shows for any biasing scheme on a random walk on a group with minimal probability of a generator $\alpha >0$, the mixing time under the bias is at worst $\frac{1}{\alpha}$ times the original walk's mixing time. For this walk, this gives a \. $\frac{1}{2a^2}N\log N$ \. upper bound, which has correct order of magnitude but not strong enough for a cutoff.

\bigskip

\section{Marking Scheme}

At time $t$, let $R_t$ be the card selected by the right hand and $L_t$ the card selected by the left hand. Let $p(R_t)$ and $p(L_t)$ be the bias of the card, either $a$ or $b$. We will construct a marking scheme so that the following holds. At all times conditioned on the time, locations of marked cards, and values of the marked cards, the marked cards should be uniformly distributed. At $t=0$ no cards are marked. Let $k$ be the number of marked cards at the beginning of the step $t$. As in Matthews, two different marking schemes will be utilized. The first scheme will be used while $k < c_1N$, and the latter between $c_1N \leq k < N$. The marking scheme ends when all cards are marked. The first phase will contribute $O(N\log\log n)$ steps and the latter $(1+\epsilon)\frac{1}{2a}N \log N$. The value of $c_1$ will depend on the choice of the $\epsilon$ in the definition of cutoff, but will always be taken to be greater than~$\frac{1}{2}$.

While $k< c_1N$, let $R_t$ and $L_t$ selected independently according to the bias. If both $R_t$ and $L_t$ are unmarked, mark $R_t$ with probability $\frac{a^2}{p(R_t)p(L_t)}$. %Informally, this scheme builds a random permutation by.... Since this method would take $O(n^2)$ time to mark the last few cards, we turn to a scheme that utilizes the already marked cards.

For $t$ with $c_1N \leq k < N$, let $m_i$ be the $i$th marked card on the table and $u_i$ the $i$th unmarked card on the table. For each $u_i$ associate one ordered pair of marked cards $(m_{r(u_i)},m_{\ell(u_i)})$ so that the mapping $(r,l)$ is injective and at least one is of the same bias ($a$ or $b$) as $u_i$. This can be accomplished if $c_1 > \frac{1}{2}$. For $1 \leq i \leq n- k$, mark $u_i$ with the given probability if any of the following happen:
\begin{enumerate}\label{s2c}
\item $R_t = L_t = u_i$
		\begin{itemize}
			\item mark $u_i$ w.p. $\frac{a}{p(u_i)}$
		\end{itemize}
\item $R_t = u_i$, $L_t$ is marked
		\begin{itemize}
			\item mark $u_i$ w.p. $\frac{a}{p(L_t)}$
			\item move the mark from $L_t$ to $u_i$ w.p. $1 - \frac{a}{p(L_t)}$
		\end{itemize}
\item $L_t = u_i$, $R_t$ is marked
		\begin{itemize}
			\item mark $u_i$ w.p. $\frac{a}{p(R_t)}$
			\item move the mark from $R_t$ to $u_i$ w.p. $1 - \frac{a}{p(R_t)}$
		\end{itemize}
\item $R_t = r(u_i)$, $L_t = \ell(u_i)$
		\begin{itemize}
			\item mark $u_i$ w.p. $\frac{ap(u_i)}{p(R_t)p(L_t)}$
		\end{itemize}
\end{enumerate}

% In this second scheme we need the marking probability of a card that's not moving to be $2a^2$ for an $a$ card and $2ab$ for a $b$ card. Since the $L_t=R_t=u_i$ gives $b^2$ and $2ab-b^2 <a^2$, the $(r,l)$ map just has to be injective. This would give case 1 acceptance probability as 1 and case 4 acceptance probability as $\frac{2ap(u_i) - p(u_i)^2}{p(R_t)p(L_t)$. Since this is pretty ugly and we were planning on $c_1$ being big anyway, I went for the cleaner probabilities.

The general outline of this scheme is inherited from Matthews with the additions of the acceptance probabilities and the moving of marks to maintain the relative uniformity of the marked cards. The first scheme can be thought of as building a random permutation by choosing a random unmarked card and putting it in a random unmarked position. The second scheme continues building the random permutation selects an unmarked card and moves it to a uniformly random position relative to the currently chosen cards. The moves that mark cards generate the uniform randomness and all additional moves by transitivity of the group, preserve it. In Section \ref{psu}, these will be proven to give a strong uniform time.

\bigskip

\section{Upper Bound}

\subsection{First marking scheme}

Let $T_k$ be the first time there are $k$ marked cards. For $k < c_1N$ with $k_a$ $a$ cards and $k_b$ $b$ cards marked, $T_{k+1}- T_{k}$ has a geometric distribution with probability of success
\[p_{k} = (n-k_a)^2 \frac{a^2}{2n^2} + 2(n-k_a)(n-k_b) \frac{a^2}{2n^2} +  (n-k_b)^2 \frac{a^2}{2n^2} = \left(\frac{a(N-k)}{N}\right)^2.
\]
%\[p_{k} = \left(\frac{\left(\frac{n}{2}-ak_1\right) + \frac{a}{b}\left(\frac{n}{2} - bk_2\right)}{n}\right)^2= \left(\frac{n-abk}{bn} \right)^2.\]
An effect of the laziness in the first marking scheme is that each card is marked at the same rate regardless of its type, so the time is independent of the current number of $a$ and $b$ cards marked.

\begin{prop}\label{tc1n}
For every constant $0 < c_1 < 1$, there exists a constant~$C$, such that:
\[ P(T_{c_1N} > N\log\log N) \leq \frac{C}{\log\log N}\]
%For any constants $\epsilon, \beta$, and $c_1 < 1$, there exists a constant $C$ such that,
%\[ P(T_{c_1n} > n\log\log n) \leq \frac{C}{\log\log n}\]
\end{prop}
\begin{proof}
The expected time to mark $c_1N$ cards is
\[
\mathbb{E}T_{c_1N}=\sum_{j=0}^{c_1N -1} \frac{1}{p_j} \leq \frac{c_1}{a^2\left(1-c_1 \right) }N,
\]
%\leq n\frac{c_1b^2}{\left(1 - c_1ab\right)^2}\]
and the statement thus follows by Markov's inequality with $C = \frac{c_1}{a^2\left(1-c_1 \right)}$.
\end{proof}

\subsection{Second marking scheme}

For $k>c_1N$, $T_{k+1}-T_k$ is still geometric, but will depend on $k_a$ and $k_b$, the current numbers of marked $a$ and $b$ cards. In addition to increasing $k_a+ k_b$, there is the possibility of incrementing $k_a$ and decrementing $k_b$. It will be simple to show that the expected time it takes for $k_a=\frac{n}{2}$ is as required. For $k_b$, since $b$ cards are unmarked at a rate depending on the current number of $a$ cards, tracking both types requires more sophisticated techniques. The changes to $k_a$ and $k_b$ form an absorbing Markov chain. Below, we compute the transition probabilities of this Markov chain $K$ on pairs $(k_a,k_b)$ of the current number of marked $a$ and $b$ cards.

The number of marked $b$ cards changes if one of three things happen. The first is that $L_t = R_t$ is a $b$ card and the $\frac{a}{b}$ chance to mark succeeds. This occurs with probability $\left(n - k_b\right) \frac{b^2 }{(2n)^2} \frac{a}{b}$. The second is that that one hand lands on a marked $a$ or $b$ card and on an unmarked $b$ card, and that the chance to mark succeeds. This has probability $\frac{2ak_ab\left(n - k_b\right) + 2bk_bb\left(n - k_b \right)\frac{a}{b}}{(2n)^2}$. Finally, the last option is that the left and right are the chosen pair to mark a $b$ card, which is designed to occur with probability $\frac{ab}{n^2}$ for each of the $\frac{n}{2}-k_b$ unmarked $b$ cards. These sum to:
\[
K\bigl( (k_a,k_b), (k_a,k_b+1)\bigr) \. = \, \frac{2ab\left(n - k_b\right)(k_a + k_b + 1)}{(2n)^2}\..
\]
The logic is the same for increasing the number of $a$ cards without reducing the number of marked $b$ cards.
\[K\bigl( (k_a,k_b), (k_a+1,k_b) \bigr) \. = \, \frac{2a^2\left(n -k_a\right)(k_a + k_b + 1)}{(2n)^2}\..
\]
The number of marked $b$ cards decreases if one hand lands on a marked $b$ card, the other on an unmarked $a$ card, and the chance to mark is unsuccessful. There are $ \left(n - k_a\right)k_b$ such pairs. The chance of picking each is \ts $\frac{2ab}{(2n)^2}$, and the chance the mark is moved is \ts $1 - \frac{a}{b}$.  We have: 
\[K\bigl( (k_a,k_b), (k_a+1,k_b-1) \bigr) \. = \, \frac{2a(b-a)\left(n - k_a\right)k_b}{(2n)^2} \.. 
\]
Thus, when there are $k_a$ marked $a$ cards, the geometric rate for the waiting time $T^{(a)}_{k_a+ 1} - T^{(a)}_{k_a}$ for a new $a$ card to be marked is
\[ p^{(a)}_{(k_a,k_b)} \. = \. \frac{2a\left(n-k_a\right)(ak_a + bk_b + a)}{(2n)^2} \..
\]
We can bound the above probability as follows. Since $c_1 > \frac 12$, $k_a + k_b \ge c_1N > n$. Further, since $a<b$, $ak_a+bk_b$ is minimized when $k_a$ is as large as possible. Setting $k_a = n$ gives $ak_a+bk_b \ge 2n(2c_1 - 1)$. Hence: 
\[ p^{(a)}_{(k_a,k_b)} = \frac{2a\left(n-k_a\right)(ak_a + bk_b + a)}{(2n)^2} \ge \frac{a(n-k_a)(2c_1-1)}{n}\..
\]

Using this estimate we can bound the expected time to mark all the $a$ cards during the second marking scheme as follows.

\begin{prop}
The expected time for all the $a$ cards to be marked is bounded as: 
\[ \mathbb{E}T^{(a)}_{n} \, \leq \, \frac{n}{a(2c_1 - 1)} \.\log (2c_1-1)n \. + \.O(n)\ts.\]
\end{prop}

Its left to find an upper bound on when all the cards are marked. Continuing in the same vein as the approximation for $p^{(a)}$, we bound the transition probabilities from below by:

\begin{equation}\label{btp}
K((k_a,k_b), (k_a',k_b')) \geq \begin{cases} \frac{ab(2c_1-1)(n- k_b)}{n} & k_a'=k_a, k_b'=k_b+1 \\
\frac{a^2(2c_1-1)(n-k_a)}{n} & k_a' = k_a+1, k_b' = k_b\\
\frac{a(b-1)(2c_1-1)(n-k_a)}{n} & k_a' = k_a+1, k_b'=k_b-1
\end{cases}
\end{equation}

Letting $s(k_a,k_b)$ be the expected amount of time to go from $k_a$ marked $a$ cards and $k_b$ marked $b$ cards to all cards marked, clearly $s(n,n)=0$. The expected time for $k_a$, $k_b$ to change is at most the reciprocal of the sum of the transition probabilities in (\ref{btp}). According to the probabilities of what kind of new mark occurs, we add the expected time from these new values.
$$
\aligned 
s(k_a,k_b) \,  \leq \, & \frac{n}{a(2c_1-1)((n-k_a) + b(n-k_b))} \. + \. \frac{b(n-k_b)}{(n-k_a) + b(n-k_b)}\. s(k_a,k_b+1)\\
& +\frac{a(n-k_a)}{(n-k_a) + b(n-k_b)}\. s(k_a+1,k_b) \. + \. \frac{(b-1)(n-k_a)}{(n-k_a) \. + \. b(n-k_b)}\. s(k_a+1,k_b-1)\ts.
\endaligned
$$

Since the fraction \. $\frac{n}{a(2c_1-1)}$ \. appears in all constant factors, we can factor it out and simplify the recurrence to:
$$\aligned
\tilde{s}(k_a,k_b) \, & = \, \frac{1}{n-k_a + b(n-k_b)}\Bigl[1 + b(n-k_b)\tilde{s}(k_a,k_b+1) \\ 
& \qquad \qquad + (n-k_a)\bigl(a\tilde{s}(k_a+1,k_b) + (b-1)\tilde{s}(k_a+1,k_b-1)\bigr)\Bigr]\ts.
\endaligned
$$
Since the first marking scheme marked cards irrespective of whether they were $a$ or $b$ cards, giving a binomial distribution to the number of marked $a$ and $b$ cards at the beginning of the second scheme, we are interested in: 
\[\sum_{(n-k_a) + (n-k_b) \. = \. (1-c_1)2n} { 2n(1-c_1) \choose n- k_a}\. 2^{-2n(1-c_1)} \. \frac{n}{a(2c_1-1)} \. \tilde{s}(k_a,k_b)\ts.\]

\noindent
Small examples indicate that, irrespective of the value of $b$, this may be exactly \ts $H\bigl((1-c_1)2n\bigr)$, 
where \ts $H(n) = 1 + \frac{1}{2} + \ldots + \frac{1}{n}$ \ts is the $n$-th harmonic number.

Instead, we will translate the problem into a new coupon collector frame work and show the $b$ cards are expected to be marked by an additional $O(n\log\log n)$ steps after the $a$ cards are marked. This approach bounds $s(k_a,k_b)$ for every starting condition for the second scheme, and does not make use of the binomial distribution of $k_a,k_b$ at the start of the second scheme. We can view the bounded transition probabilities in (\ref{btp}) as touching a single card where unmarked $b$ cards are touched with probability 
$$\frac{ab(2c_1-1)}{n}\.,$$ 
and touching each unmarked $a$ card with probability 
$$\frac{a(2c_1-1)}{n}\ts.$$ 
When touched, $b$ cards are always marked, while $a$ cards are marked with probability $a$ and turned into $b$ cards with probability $b-1$ (note that $a + b-1 = 1$). The expected time for all $(n-k_b)$ original $b$ cards to be touched is 
$$
\frac{n}{ab(2c_1-1)} \ts \bigl((\log n-k_b)+ O(1)\bigr)\ts.
$$ 
As for the $(n-k_a)$ original unmarked $a$ cards, they are marked faster than if the cards are touched with probability $\frac{a(2c_1-1)}{n}$, even after being turned into $b$ cards, and marked only after the second touch. A result of Newmann and Shepp extends the coupon collector problem with $n$ coupons to the problem of collecting $m$ of each coupon, with expected time $n\log n + (m-1)n\log\log n + O(n)$, see~\cite{NS}. Therefore, the expected time to mark all the original $a$ cards it at most 
$$\frac{n}{a(2c_1-1)}\. \bigl[\log n-k_a + \log\log n + O(1)\bigr] \ \. \text{steps.}
$$ 
Therefore, for every $k_a$ and $k_b$ at $T_{2c_1n}$, we have:
$$\mathbb{E}(T_{2n} - T_{2c_1n}) \, \le \, \frac{n}{a(2c_1-1)}\ts \bigl[\log 2n \. + \. \log\log 2n \. + \.O(1)\bigr]\ts.
$$

Note that selecting $c_1 = \frac{1}{2}\left(1 + \frac{1}{1+ \epsilon}\right)$ ensures that $\frac{1}{2c_1 - 1} \leq 1 + \epsilon$. The actual choice of $c_1$ will occur in Lemma~\ref{ubf} and accounts for the lower order terms. It remains to bound the variance of $T_N - T_{c_1N}$.

\subsection{Variance Bound}

For $k > c_1N$, the times between marking cards, $T_{k+1} - T_k$, are not independent as they depend on how many $a$ and $b$ cards are currently marked. However, when on a diagonal $k_a + k_b = k$, the rate of advancing to $k+1$ marked cards is slowest when $k_a$ is small. In steps in which no new card is marked, either an mark is moved from an $b$ card to an $a$ card, increasing all future rates, or nothing happens. This means the times $T_{k+1}-T_k$ are negative correlated.

\begin{prop}\label{varb}
We have:
\[ \Var(T_{N} - T_{c_1N}) \, \leq \, \frac{\pi^2}{6} \frac{N^2}{a^4c_1^2}\..\]
\end{prop}
\begin{proof}
Note that: 
\[ \Var(T_{N} - T_{c_1N}) \, \leq \, \sum_{k>c_1N}\Var(T_{k+1} - T_k) \]
The largest variance happens when only $a$ cards are unmarked, so $\Var(T_{k+1} - T_k) \leq \Var(Y)$ where $Y \sim Geo(p^{(a)}_{\left(k- n,n\right)})$.
Therefore, 
\begin{align*}  \Var(T_{N} - T_{c_1N}) &\leq \sum_{(c_1 - 1/2)N \leq k_a \leq n -1} \frac{(2n)^2}{a^4 (2n-k_a)^2 c_1^2}\left( 1 - \frac{a^2(2n-k_a)c_1}{2n}\right) \\
& \leq \frac{(2n)^2}{a^4c_1^2}\sum_{(c_1 - 1/2)2n \leq k_a \leq n -1} \frac{1}{k_a^2} \ \, \leq \ \, \frac{\pi^2}{6} \frac{(2n)^2}{a^4c_1^2}\.,
\end{align*}
which completes the proof. 
\end{proof}

We arrive at the upper bounded needed for cutoff.

\begin{lemma}\label{ubf}
For every $\epsilon>0$, we have:
\[ P\left(T_{N} > (1 + \epsilon)\frac{1}{2a}N\log N \right) \, = \, o(1)\ts.\]
\end{lemma}

\begin{proof} 
The Chebychev inequality, Proposition~\ref{tc1n} and the variance bound in equation~\ref{varb} give:

\[ P\left(T_{N} - T_{c_1N} > \frac{N}{(2c_1 - 1)\ts 2a}\log \frac{(2c_1-1)\ts N}{2} + C_2N + \sqrt{\frac{\pi^2}{6} \frac{N^2}{a^4c_1^2}}\. \log\log N \right) \, \leq \, \frac{1}{\log\log^2(N)}\..\]
Given $\epsilon>0$, choose $c_1$ so that,
\[ \frac{N\log N}{(2c_1 -1)\ts 2a} \. + \. \frac{N \ts \log (2c_1-1)/2 }{(2c_1-1)\ts 2a}\. + \. C_2N \. + \. 
\frac{a^2\pi}{c_1\sqrt{6}}\log\log N \. + \. \frac{C}{\log\log N} \leq \frac{(1 + \epsilon)}{2a}\. N\log N \..\]
Then:
$$P\left(T_{N} \, > \, (1 + \epsilon)\. \frac{1}{2a}\. N\log N \right) \, \leq \, \frac{2}{(\log\log n)^2}\,, 
$$
as desired. 
\end{proof}

\bigskip

\section{Lower Bound}

Let $A_K$ be the permutations of $2n$ with at least $K$ $a$ cards as fixed points. The lower bound will arise from bounding the size of $A_K$ through counting and $P^{*t}(A_K)$ using a coupon collector argument. This gives a lower bound on total variation distance as, if $U$ is the uniform distribution,

\[ ||P^{*t} - U||_{TV} \, \geq \, |P^{*t}(A_K) - U(A_K)|.\]

The original lower bound of Diaconis and Shahshahani for the transposition walk uses the permutations with no fixed points as a bad set that is less likely than it should be. Since this set has size the number of derangements of $N$, which is asymptotically $\frac{1}{e}$ of the permutations, it can only be used to show the total variation distance is bounded away from $0$ by that proportion. The argument here extends this bad set argument using the complement of $A_K$ (or $A_K$ as a too good set), using larger bad sets.

To use coupon collector, consider after $t$ steps the set of touched cards $\{R_1,L_1,...,R_{t},L_{t}\}$. The chance that each $R_t$ or $L_t$ is an $a$ card is $\frac{a}{2n}$. Let $\tau_{n-K}$ be the first time $n - K$ $a$ cards have been touched as either $R_t$ or $L_t$ in the Markov chain. Before $\tau_{n-K}$, there are at least $K$ $a$ cards in their original position, so the cards are in an arrangement in $A_K$. Using that there are ${2n \choose k} d(2n-k)$ permutations with $k$ fixed points, where 
$$d(2n) \. = \. (2n)! \. \sum_{i=0}^{2n} \. \frac{(-1)^i}{i!}$$ 
is the number of derangements of $2n$, it follows that:
\[ 
P(A_K^C) \, = \, \sum_{k=0}^{K-1} \. \frac{{2n \choose n}\. (n)! \ts {n \choose k}d(n-k)}{(2n)!} \, = \, 
\sum_{k=0}^{K-1} \. \sum_{i=0}^{n-k} \. \frac{(-1)^i}{i!k!}\,.
\]
For every constant $0<\delta<1$, by the rapid convergence of the Taylor series of $e^x$, setting $K = (2n)^\delta$ gives $P(A_K) \to 0$, as $n \rightarrow \infty$.

If we re-index $\{R_1,L_1,...,R_{t},L_{t}\} = \{C_1,...,C_{2t}\}$, let $\tilde{\tau}_{n-K}$ be the first $s$ such that $\{C_1,...,C_s\}$ contains at least $n-K$ $a$ cards. Coupon collector will be easier to state will $\tilde{\tau}$, and we can recover $\tau$ as $\tau_{n-K} = \lceil\frac{1}{2}\tilde{\tau}_{n-K}\rceil$. The difference $\tilde{\tau}_{i+1} - \tilde{\tau}_i$ are geometric with success rate $\frac{a\left(n-i\right)}{2n}$. Therefore, the expected value and variance of $\tilde{\tau}$ satisfy:
\[ \mathbb{E} \tilde{\tau}_{n-K}  \, = \, \frac{2n}{an} + \ldots + \frac{2n}{a\left(K+1\right)} \, = \, \frac{2n}{a}\left(H_{n} - H_{K}\right),\]
\[ \Var(\tilde{\tau}_{n-K})\, \leq \, \left(\frac{2n}{an}\right)^2 + \ldots + \left(\frac{2n}{a\left(K+1\right)}\right)^2 
\, \leq \, \frac{(2n)^2}{a^2}\frac{\pi^2}{6}\..
\]
Using Chebeychev's inequality, with \ts $K = (2n)^\delta$, for \ts $\delta = \frac{\epsilon}{2}$ \ts and \ts $c = \frac{\epsilon}{2}\log 2n$, this gives:

\[ P\left( |\tau_{n-n^\delta} - \frac{2n}{2a}\left( H_{n} - H_{(2n)^\delta}\right)| 
\, \geq \, \frac{\epsilon}{2}\frac{2n}{2a}\log n \right) \, \leq \, \frac{4a \pi^2}{3\epsilon^2 (\log 2n)^2} \,.
\]
Therefore, $\lim_{n \rightarrow \infty} P(\left(\tau_{n - (2n)^\delta}) < (1 - \epsilon) \frac{2n}{2a} \log 2n\right) =0$ and so for $K = (2n)^{\epsilon/2}$,  \[
\lim_{n \rightarrow \infty} P^{*(1-\epsilon) \frac{2n}{a} \log 2n}(A_k) \. = \. 1\ts,
\] while $U(A_{K}) \rightarrow 0$.
This gives the lower bound on total variation distance needed for cutoff.

\bigskip

\section{Proof of Strong Uniformity}\label{psu}
As in Matthews's original analysis, the proof of strong uniformity will be divided into two parts.
To analyze the first marking scheme, we will track the values and locations of the marked cards separately. Under this first scheme, the randomness in the list of values of the marked cards (ordered by time of marking) is the primary source of uniformity. Next we will show, when the cards are being marked in either of the schemes, the marginal of the marked cards is invariant under permutations. The latter property implies that the first time all cards have been marked is a strong uniform time.

Writing $\pi_t = (R_t L_t) (R_{t-1} L_{t-1} )\cdots (R_1 L_1)$ where multiplication is from right to left, this is a map from locations to values after $t$ steps of the walk. As in Matthews's original proof, we will track the marked cards using two permutations in $S_n$, $\phi_t$ and $\psi_t$ which will denote the labels and positions, respectively, in order of marking of the marked cards (and an order for the remaining cards to be defined), such that $\pi_t = \phi_t \psi_t^{-1}$. If $k$ cards have been marked at time $t$, $(\phi(1),\ldots,\phi(k))$ will be the labels of the marked cards in the order they were marked in and $(\psi(1),\ldots,\phi(k))$ their locations. We will use the same choice of $\phi$ and $\psi$ as in Matthews's proof, except for a modification to keep $\phi_t$ constant whenever a new card is not marked.

 Matthews's orginal proof showed that both of these order $k$-tuples are uniformly distributed independent subsets of $[2n]$ of size $k$. This is no longer true in the biased case considered here, as the locations will be biased by the non-marking steps of the walk. Instead, we will show $(\phi(1),\ldots,\phi(k))$ is a uniformly distributed subset of $[2n]$ of size $k$, and $(\phi(1),\ldots,\phi(k))$ and $(\psi(1),\ldots,\psi(k))$ are independent.

The walk then maps by \ts $\pi_t = \phi \psi^{-1}$ \ts the locations \ts $(\psi(1),\ldots,\psi(k))$ \ts to the cards labeled \ts 
$(\phi(1),\ldots,\phi(k))$. Since these lists are independent, fixing \ts $(\psi(1),\ldots,\psi(k))$, does not change the distribution of $\phi$. For values of $\{\phi(1),\ldots,\phi(k)\}$ a fixed $k$-subset of $[2n]$, each of the orders are equally likely. Therefore, if the assumption on the distributions of $\phi$ and $\psi$ hold, the permutation of marked card positions to values is uniformly distributed for each set of positions and values.

\begin{prop}\label{s1sst}
For each $t$ such that $k = k_t < c_1N$, there exist permutations of $[2n]$, $\phi_t$ and $\psi_t$ such that $\pi_t  = \phi_t \psi_t^{-1}$. Further, $(\phi(1),\ldots,\phi(k))$ is a uniformly distributed $k$-tuple of $[2n]$, and $(\phi(1),\ldots,\phi(k))$ and $(\psi(1),\ldots,\psi(k))$ are independent.
\end{prop}
\begin{proof}

For $t=0$, $\psi_0 = \phi_0 = \pi_0 = \id$ with $k_0 =0$, so the statement holds vacuously.

For $t>0$, assume by induction that the claim is true up to time $t-1$. Below, we will define $\phi_t$, $\psi_t$ in terms of $\phi_{t-1}$ and $\psi_{t-1}$. The first case we consider is that $k= k_t>k_{t-1}$, that is, a new card was marked, meaning $R_t$ and $L_t$ are unmarked cards at time $t-1$ and a coin flip with probability of success $\frac{a^2}{p(R_t)p(L_t)}$ succeeded. Let $R_t^{*} = \phi_{t-1}^{-1}(R_t) $ and $ L_t^* = \phi_{t-1}^{-1}(L_t)$. Since $\phi_{t-1}([k-1])$ were the marked cards at time $t-1$, $R_t^{*}, L_t^*  >k_{t-1}$. Given that the marking succeeded,
$R_t$ and $L_t$ are uniformly and independently distributed on the unmarked cards since the chance of the marking succeeding is inversely proportional precisely to the product of the probabilities of choosing the unmarked cards. Since $\phi_{t-1}$ is a fixed permutation. Further, this implies $R_t^*$ and $L_t^*$ are uniformly and independently distributed on $\{k,\ldots,2n\}$. We have:
\[ \pi_t \. = \. (R_t L_t) \phi_{t-1} \psi_{t-1}^{-1} \. = \. \phi_{t-1} (R_t^* L_t^*) \psi_{t-1}^{-1}\ts.
\]

Define $\psi_t = \psi_{t-1}(k L_t^*)$. If $L_t^*$ or $R_t^*$ is $k$, write $(R_t^* L_t^*) = (k R_t^*)(k L_t^*)$, 
and so form $\phi_t = \phi_{t-1}(k R_t^*)$. Otherwise, $(R_t^* L_t^*) = (k R_t^*)(R_t^* L_t^*) (k L_t^*)$, 
and so form $\phi_t = \phi_{t-1}(k R_t^*)(R_t^* L_t^*)$.

The first $k-1$ values of $\phi$ and $\psi$ are unchanged at $t$ versus $t-1$, with $\phi_t(k) = R_k$, $\psi_t(k) = L_k^*$. The uniformity and independence of $R_t$ and $L_t$ along with the induction hypothesis suffice to show the lists for $\phi$ and $\psi$ have the desired properties.

When a new card is not marked, it breaks into three cases of whether two marked cards were moved, one marked and one unmarked, or two unmarked with a failed marking. In all cases let $\phi_t = \phi_{t-1}$ and $\psi_t = \psi_{t-1}(R_t^* L_t^*)$. Clearly $(\phi_t(1),\ldots,\phi_t(k))$ is still uniformly distributed since it is the same as for $t-1$. It remains to show the desired independence between $(\phi_t(1),\ldots,\phi_t(k))$ and $(\psi_t(1),\ldots,\psi_t(k))$ in each case.

If $R_t$ and $L_t$ are both marked cards, they both appear in $(\phi_{t_1}(1),...,\phi_{t-1}(k))$. By the uniformity of the distribution of $\phi$,  $R_t^{*} = \phi_{t-1}^{-1}(R_t)$ and $L_t^{*} = \phi_{t-1}^{-1}(L_t)$ are i.i.d picks from $[k]$. This acts on $\{\psi_t(1),\ldots,\psi_t(k)\}$ as a uniformly random transposition having removed the bias. This does not affect the independence between the sequences.

If one card is marked and the other is unmarked, without loss of generality, assume $R_t$ is marked, and $L_t$ is unmarked. Then $R_t^*$ is an uniform choice from $[k]$ and $L_t^* \in \{k+1,...n\}$ and the two are independent. The permutation $\psi_t = \psi_{t-1} (L_t^* R_t^*) = (\psi(L_t^*) \psi(R_t^*)) \psi_{t-1}$ replaces $\psi_{t-1}{R_t^*}$ with $\psi_{t-1}(L_t^*)$ in the list $(\psi(1),...,\psi(k))$. Since $R_t^*$ and $L_t^*$ are independent, the list is still independent from the list for $\phi$.

If $R_t$ and $L_t$ are both unmarked cards, then $R_t^*, L_t^* \notin [k]$ and $\psi(R_t*), \psi(L_t*) \notin \{\psi(1),...\psi(k)\}$, and the sequence of the first $k$ values of $\psi$ is unchanged by appending $(R_t^* L_t^*)$ to the right of $\psi$.

\end{proof}

\begin{prop}\label{s2sst}
Given that $k$ cards are marked at time $t$, the marginal distribution of the marked cards is invariant under permutation.
\end{prop}
\begin{proof}
For $k < c_1N$, this holds by Proposition \ref{s1sst}.

Under the second marking scheme, one of three things can happen.

{\it Case 1}: A new card is marked, no mark is removed.

{\it Case 2}: No new card is marked.

{\it Case 3}: A new card is marked, and the mark is removed from a card. This is also called moving a mark.

For {\it Case 1}, with equal probability the newly marked card is transposed with either of the marked cards, or remains where it is. If $\pi$ is uniformly distributed on permutations of $[k]$ with $k$ as a fixed point, $\frac{1}{k}( (1 k) + (2 k) + ... + ( k k))\pi$ is uniformly distributed on permutation of $[k]$. The newly marked card is acting as $k$ here, and so the marginal distribution of the marked cards remains uniform.

Under {\it Case 2}, either two unmarked cards were exchanged or two marked cards were exchanged (without $R_t = r(u), L_t = \ell(u)$ for all unmarked~$u$). The first trivially does not change the marginal distribution of the marked cards. Note that for a fixed permutation $\omega$, and $\pi$ is distributed uniformly in $S_n$, then so is $\omega \pi$, even conditioned on $R_t$ and~$L_t$. Thus, the marginal distribution of the marked cards is still uniform.

Finally, {\it Case 3} occurs if one of $R_t$ and $L_t$ is a marked $b$ card and the other is unmarked with probability $\frac{b-a}{b}$. This moves the mark to the previously unmarked card and puts that card in to the same location as the previously marked card. Since by induction, any ordering of the marked cards was equally likely, the newly marked card assumes the place of the previously marked card in each of these equally likely orders, and the same property holds.
\end{proof}

\bigskip

\section{Final remarks}\label{fin}

The construction in the paper can in principle be modified to work for every 
``biased'' distribution on permutations with probability of every transposition
$\Theta(1/N^2)$.  It would be interesting to see how far this bound can be pushed.
For example, is there a cutoff for probability of $(i,j)$ proportional to $(j-i)^3$. 

Another possible direction for generalization is the many examples of random walks
on matrix groups \ts SL$(n,q)$, \ts SO$(n,\Bbb R)$, etc.  Is there a reasonable way
to make a bias which would lead to the cutoff?  

\vskip.8cm

\subsection*{Acknowledgements}
We are very grateful to Nathanael Berestycki and Justin Salez for many helpful discussions.
This work originated while the authors were at the
``Markov chain mixing times" workshop at the American Institute
of Mathematics.  We would like to thank both the institute and
the workshop organizers for their help and encouragement.
The first author was partially supported by~NSF Grant 1344199. 
The second author was partially supported by an NSF CAREER Grant 1554783 
and a Sloan Research Fellowship. The third author was partially 
supported by the~NSF Grant~1363193.

% \vskip.7cm

\end{document}